\documentclass[letterpaper, 10 pt, conference]{ieeeconf}  

\IEEEoverridecommandlockouts                              
\usepackage[letterpaper,
            left=.75in,
            right=.75in,
            top=1in,
            bottom=.75in]{geometry}

\usepackage{cite}
\usepackage{amsmath,amssymb,amsfonts}
\usepackage{algpseudocode}
\usepackage{algorithm}
\usepackage{graphicx}
\usepackage{textcomp}
\usepackage[abs]{overpic}
\usepackage{mathrsfs}
\usepackage{xcolor}
\definecolor{violet}{rgb}{.5,0,.5}
\definecolor{orange}{rgb}{1,.65,0}

\newcommand{\inner}[1]{\langle #1 \rangle}
\newcommand{\lift}{\text{lift}}

\newcommand{\tr}{\text{tr}}

\newcommand{\grad}{\text{grad}}
\newcommand{\Hess}{\text{Hess}}
\newcommand{\co}{{\min}}
\newcommand{\KM}{\mathrm{KM}}
\newcommand{\GL}{\mathrm{GL}}

\newcommand{\lambdamax}{\overline{\lambda}}

\newcommand{\cl}{\text{cl}}

\newcommand{\calC}{\mathcal{C}}
\newcommand{\calD}{\mathcal{D}}

\newcommand{\calG}{\mathcal{G}}
\newcommand{\calH}{\mathcal{H}}

\newcommand{\calK}{\mathcal{K}}
\newcommand{\calL}{\mathcal{L}}
\newcommand{\calM}{\mathcal{M}}
\newcommand{\calN}{\mathcal{N}}
\newcommand{\calO}{\mathcal{O}}

\newcommand{\calR}{\mathcal{R}}
\newcommand{\calS}{\mathcal{S}}

\newcommand{\calU}{\mathcal{U}}
\newcommand{\calV}{\mathcal{V}}

\newcommand{\scrT}{\mathcal{T}}
\newcommand{\sfK}{\mathsf{K}}
\newcommand{\sfL}{\mathsf{L}}

\newcommand{\bbL}{\mathbb{L}}

\newcommand{\bbN}{\mathbb{N}}

\newcommand{\bbR}{\mathbb{R}}
\newcommand{\bbS}{\mathbb{S}}

\newcommand{\bfE}{\mathbf{E}}
\newcommand{\bfF}{\mathbf{F}}
\newcommand{\bfG}{\mathbf{G}}

\newcommand{\bfV}{\mathbf{V}}
\newcommand{\bfW}{\mathbf{W}}

\newcommand{\mat}[1]{\begin{bmatrix} #1 \end{bmatrix}}

\newtheorem{theorem}{Theorem}[section]

\newtheorem{lemma}[theorem]{Lemma}

\newtheorem{assumption}[theorem]{Assumption}

\def\BibTeX{{\rm B\kern-.05em{\sc i\kern-.025em b}\kern-.08em
    T\kern-.1667em\lower.7ex\hbox{E}\kern-.125emX}}
\begin{document}
\title{Output-feedback Synthesis Orbit Geometry:\\
Quotient Manifolds and LQG Direct Policy Optimization}
\author{Spencer Kraisler, \IEEEmembership{Member, IEEE} and Mehran Mesbahi, \IEEEmembership{Fellow, IEEE} 
\thanks{The authors are with the William E. Boeing Department of Aeronautics and Astronautics, University of Washington. Emails: \{kraisler+mesbahi\}@uw.edu.}}

\maketitle

\begin{abstract}
We consider direct policy optimization for the
linear-quadratic Gaussian (LQG) setting. Over the past few
years, it has been recognized that the landscape of 
dynamic output-feedback controllers of relevance to LQG has an intricate
geometry, particularly pertaining to the existence of degenerate stationary points, that hinders gradient methods. In order to address these challenges, in this paper, we adopt a system-theoretic coordinate-invariant Riemannian metric for the space of dynamic output-feedback controllers and develop a Riemannian gradient descent for direct LQG policy optimization. We then proceed to prove that the orbit space of such controllers, modulo the coordinate transformation, admits a Riemannian quotient manifold structure. This geometric structure--that is of independent interest--provides an effective approach to derive direct policy optimization algorithms for LQG with a local linear rate convergence guarantee. Subsequently, we show that the proposed approach exhibits significantly faster and more robust numerical performance as compared with ordinary gradient descent.
\end{abstract}

\section{Introduction}
\label{sec:introduction}
Direct policy optimization (PO) synthesizes controllers by formalizing optimization problems over controller parameters rather than first solving for value functions or Lyapunov certificates, say using matrix inequalities. In recent years, PO has been shown to be an effective first-order procedure for a number of feedback synthesis problems, while also providing a natural bridge between control synthesis and reinforcement learning with \textit{stabilization} guarantees \cite{hu2023toward}. 

In the PO setting, design problems such as the linear-quadratic regulator (LQR), linear-quadratic Gaussian (LQG), and mixed $H_2/H_\infty$, are directly parameterized in terms of the corresponding stabilizing feedback parameters; subsequently, a first-order method is adopted to update these parameters with the goal of local convergence guarantees to the optimum. Such a ``direct'' synthesis procedure has necessitated a deeper understanding of the interplay between analytic properties of various control design objectives
in relation to the geometry of the space of stabilizing controllers~\cite{hu2022connectivity, bu2019topological, bu2020topological, tang2021analysis}. 

In the context of PO--as it turns out--LQG and the domain of dynamic output-feedback controllers has a more intricate landscape 
as compared with LQR and the domain of static state-feedback controllers, hindering a straightforward
adoption of first order methods. This includes spurious non-strict saddle points \cite{zheng2022escaping} and no coercivity. In fact, to the best of our knowledge, there are no local convergence guarantees for LQG PO. There are a number
of reasons for this. First, the coordinate-invariance of LQG implies that each stationary point lies within an ``orbit'' of stationary points. This observation implies that the LQG cost admits stationary points with singular Hessians. Next, certain systems can admit \textit{degenerate} LQG controllers, greatly impacting the convergence rate of optimization algorithms. In fact,
LQG controllers can be non-minimal and have small stability margins. Lastly, the search space of full-order LQG controllers is large with $n^2 + nm + np$ dimensions, where $n$, $m$, and $p$ are the dimensions of the state, control, and output, respectively. 

In this paper, we present a geometric framework to address the aforementioned issues; the key ingredient of our approach is framing PO for LQG over the \textit{Riemannian quotient manifold} of full-order minimal dynamic output-feedback controllers modulo coordinate transformation. Specifically, we prove that this setup is well-defined and leads to a Riemannian gradient descent (RGD) algorithm for LQG. Note that equipping a search space of an optimization problem with a Riemannian metric for the purpose of developing efficient algorithms for its solution is an active area of research in systems theory, optimization, and statistical learning~\cite{amari1998natural,mishra2016riemannian,talebi2022policy}. We show that our technique is far faster than gradient descent (GD) by invoking \cite[Thm. 4.19]{boumal2023introduction} 
for a proof of guaranteed local linear convergence under a reasonable assumption on the degeneracy of the LQG controller.

Although PO for control synthesis with implicit or explicit requirements on stabilization is a relatively recent research direction in control theory, our work benefits from tools historically developed in geometric system theory pertaining
to orbit spaces of linear systems. In fact, examining such orbits was initiated by Kalman and Hazewinkel in the early 1970s~\cite{kalman1974algebraic, hazewinkel1976moduli,hazewinkel1980fine,tannenbaum2006invariance,afsari2017bundle} in the context of system identification and realization theory. In this paper, we show that these tools are rather powerful for PO and data driven control since the set of output-feedback controllers is in fact {\em a family} of linear systems. In the meantime, optimization over
the geometry induced by orbits of linear (dynamic) controllers comes hand in hand with a number of technical issues 
that are addressed in this work.

\section{Preliminaries and Notation}\label{sect:background}

Consider the continuous-time linear system,
\begin{align}
    \dot{x}(t) = A x(t) + B u(t) + w(t), \; y(t) = C x (t) + v(t),\label{plant}
\end{align} 
\newgeometry{top=.75in,bottom=.75in,right=.75in,left=.75in,top=.75in}
as the plant model, where $A \in \bbR^{n \times n}$, $B \in \bbR^{n \times m}$, and $C \in \bbR^{p \times n}$.

The process $w(\cdot)$ and measurement $v(\cdot)$ noise terms are zero-mean Gaussian with covariance matrices $W \in \bbS^n_+$ (positive semidefinite) and $V \in \bbS_{++}^p$ (positive definite), respectively. We also assume that $(A,B)$ and $(A,W^{1/2})$ are controllable and $(A,C)$ is observable. An output-feedback (dynamic) controller of order $1 \leq q \leq n$ for the plant (\ref{plant}) is now parameterized as,
\begin{align}
    \dot{\xi}(t) = A_\sfK \xi(t) + B_\sfK y(t), \; u(t) = C_\sfK \xi(t), \label{controller}
\end{align} where $A_\sfK \in \bbR^{q \times q}$, $B_\sfK \in \bbR^{q \times p}$ and $C_\sfK \in \bbR^{m \times q}$. Let $\widetilde{\calC}_q$ be the set of all such $q$th-order output-feedback controllers, represented as 
\begin{align}\label{block-form}
    K =\mat{0_{m \times p} & C_\sfK \\ B_\sfK & A_\sfK} \in \bbR^{(m + q) \times (p + q)} .
\end{align} Now, let $J_q:\widetilde{\calC}_q \to \bbR$ be the LQG cost \cite[Eq. 2]{tang2021analysis} of $q$th-order controllers, with $Q \in \bbS^n_+$ and $R \in \bbS_{++}^m$ as the state and control cost matrices.
We assume that $(A,Q^{1/2})$ is observable. Our goal is to minimize $J_n$ over $\widetilde{\calC}_n$ via a first order procedure.

In order to examine output-feedback synthesis, we first consider the combined plant/controller closed loop
as,
\begin{subequations}\label{closed-loop}
    \begin{align}
        \mat{\dot{x} \\ \dot{\xi}} &= \mat{A & B C_\sfK \\ B_\sfK C & A_\sfK} \mat{x \\ \xi} + \mat{I_n & 0_{n \times p} \\ 0_{q \times n} & B_\sfK} \mat{w \\ v} \\
        \mat{y \\ u} &= \mat{C & 0_{p \times q} \\ 0_{m \times n} & C_\sfK}\mat{x \\ \xi} + \mat{0_{p \times n} & I_p \\ 0_{m \times n} & 0_{m \times p}}\mat{w \\ v}.
    \end{align}
\end{subequations} 

The realized closed-loop system and observation matrices are now, respectively,
\begin{align*}
    A_\cl(\sfK) &\in \bbR^{(n + q) \times (n + q)}, & B_\cl(\sfK) \in \bbR^{(n + q) \times (n + p)}, \\
    C_\cl(\sfK) &\in \bbR^{(m + p) \times (n + q)}, & D_\cl(\sfK) \in \bbR^{(m + p) \times (n + p)}.
\end{align*} Hence, (\ref{controller}) a is stabilizing feedback when $A_\cl(\sfK) \in \calH_{n + q}$, where $\calH_k$ denotes the set of $k \times k$ Hurwitz stable matrices. 

Let $\widetilde{\calC}^\co_q$ be the set of minimal (i.e., controllable and observable) $q$th-order output-feedback controllers. Our first observation is as follows.
\begin{lemma}\label{lem:generic}
    The subset $\widetilde{\calC}^\co_q \subset \widetilde{\calC}_q$ is an open, dense subset with a measure zero complement.
\end{lemma}
\begin{proof}
    This follows from the openness of $\widetilde{\calC}_q$ \cite[Lem. 2.1]{tang2021analysis} and the genericity of controllability and observability for linear systems \cite[Thm. 1.3]{wonham1974linear}.
\end{proof}

A key construct for our geometric approach is 
the Lyapunov operator $\bbL(A,Q) \in \bbS_+^k$, mapping $A \in \calH_k$ and $Q \in \bbS_{+}^k$ to the unique solution of $AP + PA^\textsf{T}  =-Q$. In fact, defining the maps
\begin{subequations}
    \begin{align}
        Q_\cl(\sfK) &:= \mat{Q & 0_{n \times q} \\ 0_{q \times n} & C_\sfK^\textsf{T}  R C_\sfK}, \\
        W_\cl(\sfK) &:= \mat{W & 0_{n \times q} \\ 0_{q \times n} & B_\sfK V B_\sfK^\textsf{T} }, \\
        X(\sfK) &:= \bbL\left(A_\cl(\sfK), W_\cl(\sfK) \right) \label{X},
    \end{align}
\end{subequations} on $\widetilde{\calC}_q$, facilitates recognizing that $J_q(\sfK) = \tr\left(Q_\cl(\sfK) X(\sfK)\right)$ is the LQG cost.
We note that the Euclidean gradient and Hessian of $J_q$ have been characterized in \cite{tang2021analysis}.

\subsection{Geometry of Riemannian gradient descent}\label{sect:rie-geo}
The algorithmic framework examined in this paper for PO synthesis requires basic notions from Riemannian geometry; these are briefly reviewed in this section. 
Let $\calM \subset \bbR^N$ be a smooth manifold. A smooth curve is a smooth function $c:\bbR \to \calM$. The tangent space at $x \in \calM$, denoted as $T_x \calM$, is the set of the tangent vectors $\dot{c}(0)$ of all smooth curves $c(\cdot)$ with $c(0)=x$. For example, as an Euclidean open set, the tangent spaces of $\widetilde{\calC}_q$ identifies as
\begin{equation*}\label{tangent-space}
    \calV_q := \left\{\mat{0_{m \times p} & G \\ F & E} \in \bbR^{(m + q) \times (p + q)}\right\} = T_\sfK \widetilde{\calC}_q .
\end{equation*} For matrix manifolds, tangent vectors are matrices; such vectors are denoted by boldface letters, e.g., $\bfV$. The disjoint union of tangent spaces is called the tangent bundle, denoted as  $T \calM$. 

Let $F: \calM \to \calN$ be a smooth function between two smooth manifolds $\calM$ and $\calN$.
The differential $dF_x:T_x\calM \to T_{F(x)} \calN$ of $F$ at $x$ along $v \in T_x\calM$ is the linear mapping defined as
\begin{equation*}
    dF_x(v) := \frac{d}{dt}\Big|_{t=0} (F \circ c)(t),
\end{equation*} where $c(\cdot)$ is any smooth curve satisfying $c(0)=x$ and $\dot{c}(0)=v$. 
For example, the differential of the Lyapunov operator $\bbL$ is,
\begin{equation*}
    d \bbL_{(A,Q)}(\bfV, \bfW) = \bbL(A, \bfV \bbL(A,Q) + \bbL(A,Q) \bfV^\textsf{T}  + \bfW).
\end{equation*} When the differential $dF_x(\cdot)$ is independent of $x$, we  simply write $dF(\cdot)$. 

A Riemannian metric is a smooth ``state-dependent'' inner product $\inner{.,.}_x: T_x \calM \times T_x \calM \to \bbR$. This metric induces a state-dependent norm $\|v\|_x := \sqrt{\inner{v,v}_x}$. Given the open set $U \subset \calM$, a local frame is a set of linearly independent vector fields $(E_i:U \to T\calM)_{i=1}^{\dim{\calM}}$. The coordinates $G(x) \in \bbS_{++}^n$ of the metric at $x \in U$ with respect to $(E_i)$ are
\begin{equation}\label{g-coords}
    G_{ij}(x) = \inner{E_i|_x, E_j|_x}_x.
\end{equation} With $G^{ij} := (G^{-1})_{ij}$, the gradient of $f:\calM \to \bbR$ at $x$ is then,
\begin{equation}\label{Riemannian-gradient}
    \nabla f(x) = \sum_{i=1}^n \sum_{j=1}^n G^{ij}(x) df_x(E_i|_x) E_j|_x.
\end{equation} We will denote the Euclidean gradient as $\grad f$ and the Riemannian gradient as $\nabla f$; similarly, $\Hess f$ and $\nabla^2 f$ for the corresponding Hessians. 

A retraction is a smooth mapping $\calR:\calS \subset T\calM \to \calM$ where $\calS$ is open, $(x,0_x) \in \calS$ for all $x$, and the curve $c(t):=\calR_x(tv) \equiv \calR(x,tv)$ satisfies $c(0)=x$ and $\dot{c}(0)=v$ for each $(x,v) \in \calS$. We define $\calR$-balls as $B_x(\rho):=\{\calR_x(\xi):\|\xi\|_x < \rho\}$. Retractions are the central constructs in Riemannian optimization. When $\calM \subset \calV_n$, one can use the Euclidean metric and retraction,
\begin{equation}
        \inner{\bfV,\bfW}:= \tr(\bfV^\textsf{T} \bfW),  \quad
        \calR_\sfK(\bfV) := \sfK + \bfV. \label{euclidean-retraction} 
\end{equation} 

With these basic ingredients of Riemannian optimization in place, the 
Riemannian Gradient Descent (RGD) of $f$ under $(\inner{.,.}, \calR)$ is defined as,
\begin{equation}\label{RGD}
    x_{t+1} := \calR_{x_t}(-s_t \nabla f(x_t)),
\end{equation} where $s_t \geq 0$ is a chosen step-size. Figure \ref{fig:RGD} is a visual depiction of the RGD procedure.
We recommend~\cite{boumal2023introduction, absil2008optimization} for optimization-oriented references on Riemannian manifolds; see also~\cite{lee2012smooth,lee2018introduction}.

For $S \in \GL_q$ (invertible $q\times q$ matrices), define the diffeomorphism $\scrT_S: \widetilde{\calC}_q \to \widetilde{\calC}_q$,
\begin{equation}\label{coord-trans}
    \scrT_S(\sfK) = \mat{0_{m \times p} & C_\sfK S^{-1} \\ SB_\sfK & SA_\sfK S^{-1}}.
\end{equation} We call (\ref{coord-trans}) a coordinate transformation due to the change of coordinates of the controller state: $\xi=S\eta$ (\ref{controller}). Abusing the notation, we have $d\scrT_S(\bfV) = \scrT_S(\bfV)$ since (\ref{coord-trans}) is linear.
A function $F$ on $\widetilde{\calC}_q \times \calV_q$ is called coordinate-\textit{invariant} when $F(\sfK, \bfV) = F(\scrT_S(\sfK), d\scrT_S(\bfV))$.

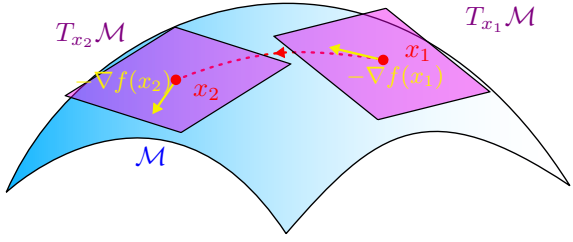
\begin{figure}
    \centering
    \begin{overpic}[width=.9\linewidth]{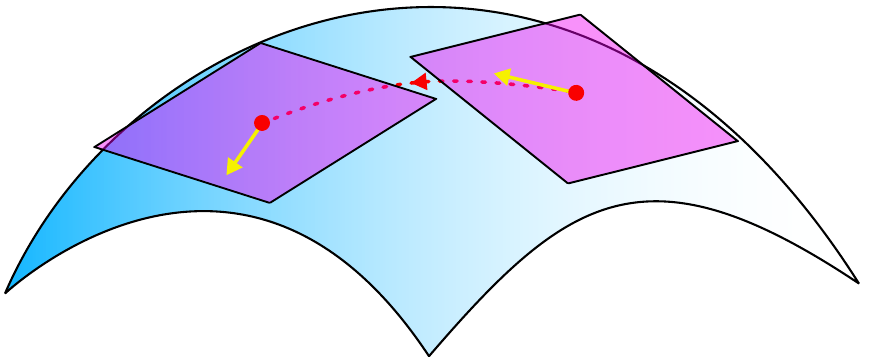}
        \put(152,68){\color{red} $x_1$}
        \put(72,53){\color{red} $x_2$}
        \put(130,58){\color{yellow} \small $-\nabla f(x_1)$}
        \put(27,56){\color{yellow} \small $-\nabla f(x_2)$}
        \put(175,80){\color{violet} $T_{x_1} \calM$}
        \put(20,74){\color{violet} $T_{x_2} \calM$}
        \put(50,27){\color{blue} $\calM$}
    \end{overpic}
    \caption{Visualization of RGD; here, $x_2 = \calR_{x_1}(-s_1\nabla f(x_1))$.}
    \label{fig:RGD}
\end{figure}

\section{Direct PO for LQG}\label{sect:alg}

In this section, we introduce our proposed first-order optimizer for LQG. Algorithm \ref{alg:RGD} is RGD over the domain of $\widetilde{\calC}_n^\co$. The optimizer uses the Euclidean retraction and the Riemannian metric defined in \S\ref{sect:KM-metric}. We optimize over $\widetilde{\calC}_n^\co$ instead of $\widetilde{\calC}_n$ for two reasons. First, the orbit space of $\widetilde{\calC}^\co_n$ modulo coordinate transformation admits a \textit{quotient manifold structure}. This is not the case for $\widetilde{\calC}_n$, whose quotient space is non-Hausdorff \cite{hazewinkel1980fine}. Second, the metric is coordinate-invariant, and hence 
\begin{equation}\label{coord-equiv}
    \scrT_S(\sfK -  s \nabla J_n(\sfK)) = \scrT_S(\sfK) - s \nabla J_n(\scrT_S(\sfK)).
\end{equation} This also implies RGD over the $(n^2+nm+np)$-dimensional Riemannian manifold $\widetilde{\calC}^\co_n$ coincides with RGD over the much smaller $(nm+np)$-dimensional \textit{Riemannian quotient manifold}. This will be explained in detail in \S\ref{sect:rie-quo-man}.

\begin{algorithm}
    \small
    \begin{algorithmic}
        \caption{Riemannian Gradient Descent}\label{alg:RGD}
        \Require $\sfK_0 \in \widetilde{\calC}_n^\co$, $\epsilon > 0$, $T \in \bbN$, $s_t \geq 0$
        \State $\sfK \gets \sfK_0$, $t \gets 0$
        \While{$t \leq T$ \textbf{and} $\|\nabla J_n(\sfK)\|_\sfK \geq \epsilon$}
            \State $\sfK \gets \sfK - s_t \nabla J_n(\sfK)$
            \State $t \gets t + 1$
        \EndWhile
        
        \Return $\sfK$
    \end{algorithmic}
\end{algorithm}

A few remarks are in order. In the case where $\sfK^+ = \calR_\sfK(-s \nabla J_n(\sfK))$ is non-stabilizing, one has to choose a small enough positive $s_t$ in Algorithm \ref{alg:RGD}; 
analogously, when $\sfK^+$ is stabilizing yet non-minimal, one can perturb the step direction since the set difference $\widetilde{\calC}_n - \widetilde{\calC}_n^\co$ is measure zero (Lemma \ref{lem:generic}). Second, $\|.\|_\sfK$ is the Riemannian norm to be defined. Next, $\nabla J_n(\sfK)$ is computed via \eqref{g-coords} and (\ref{Riemannian-gradient}) in each iteration. It should be noted that given $\grad J(\sfK)$, computing $\nabla J(\sfK)$ is an $\calO((n^2 + nm + np)^3)$ operation; Cholesky decomposition can significantly reduce the per-iteration complexity. In this context, the global frame $(\bfE_i)$ is simply  a fixed basis of $\calV_n$.
 
Lastly, the differential of the LQG cost is,
\begin{equation*}
    d J_n|_\sfK(\bfV) = \tr \left(dQ_\cl|_\sfK(\bfV) X(\sfK) +Q_\cl(\sfK) dX_\sfK(\bfV) \right),
\end{equation*} where,
\begin{align*}
    dQ_\cl|_\sfK(\bfV) &= \mat{0_{n \times n} & 0_{n \times n} \\ 0_{n \times n} & G^\textsf{T}  R C_\sfK + C_\sfK^\textsf{T}  R G}, \\
    dW_\cl|_\sfK(\bfV) &= \mat{0_{n \times n} & 0_{n \times n} \\ 0_{n \times n} & F VB_\sfK^\textsf{T}  + B_\sfK V F^\textsf{T} }, \\
    dX_\sfK(\bfV) &= d \bbL_{\left(A_\cl(\sfK), W_\cl(\sfK)\right)}(d A_\cl(\bfV), d W_\cl|_\sfK(\bfV)). \label{dX} 
\end{align*}

\subsection{Krishnaprasad-Martin Metric}\label{sect:KM-metric}

Let $\sfK \in \widetilde{\calC}^\co_n$ and $\mathbf{V} \in \calV_n$ be a tangent vector. Define,
\begin{align*}
    \bfE(\bfV) &:= \mat{0_{n \times n} & B G \\ F C & E}, &
    \bfF(\bfV) &:=  \mat{0_{n \times n} & 0_{n \times p} \\ 0_{n \times n} & F} \\
    \bfG(\bfV) &:= \mat{0_{p \times n} & 0_{p \times n} \\ 0_{m \times n} & G}.
\end{align*} Next, let $W_c(\sfK)$ and $W_o(\sfK)$ denote the controllability and observability Grammians of $(A_\cl(\sfK), B_\cl(\sfK),C_\cl(\sfK))$. Consider now the following Riemannian metric,
\begin{subequations}\label{KM-metric}
    \small
    \begin{align}
        \inner{\bfV_1,\bfV_2}_\sfK^\KM := &c_1 \, \tr[W_o(\sfK) \,\bfE(\bfV_1) \, W_c(\sfK) \, \bfE(\bfV_2)^\textsf{T} ] \\
        &+c_2 \, \tr[\bfF(\bfV_1)^\textsf{T} \, W_o(\sfK) \, \bfF(\bfV_2)] \\
        &+ c_3 \, \tr[\bfG(\bfV_1)\, W_c(\sfK) \, \bfG(\bfV_2)^\textsf{T} ],
    \end{align}
\end{subequations} where $c_1>0$, $c_2,c_3 \geq 0$ are constants. 

This metric was derived from a similar setup in \cite{krishnaprasad1983families,krishnaprasad1977geometry} for stable linear systems. In literature, the original metric is called the Krishnaprasad-Martin (KM) metric \cite{afsari2017bundle}. Although we have slightly augmented the KM metric in this work, we will keep the original name.

\subsection{Coordinate-invariance of the KM Metric}\label{sect:coord-inv-KM-metric}

In this section, we will prove that the above mentioned properties of the KM metric. 

\begin{lemma}\label{lem:closed-loop}
    For the system $(A,B,C)$ in (\ref{plant}) and $\sfK \in \widetilde{\calC}_q^\co$, the triplet $(A_\cl(\sfK), B_\cl(\sfK), C_\cl(\sfK))$ is minimal. 
\end{lemma}
\begin{proof}
    This follows from the Popov-Belevitch-Hautus test and is omitted for brevity; see also \cite[Lem. 4.5]{tang2021analysis}.
\end{proof}
\begin{theorem}\label{thm:KM-metric}
    The mapping defined in (\ref{KM-metric}) is a Riemannian metric and coordinate-invariant.
\end{theorem}

\begin{proof}
    We first note that for any $\sfK$, the mapping $\inner{.,.}_\sfK^\KM$ is smooth, bi-linear, and symmetric. Therefore, it suffices to show its positive-definiteness. By Lemma \ref{lem:closed-loop} and \cite[Thm. 12.4]{hespanha2018linear}, the Grammians satisfy,
    \begin{align*}
        W_c(\sfK) &= \bbL(A_\cl(\sfK), B_\cl(\sfK)B_\cl(\sfK)^\textsf{T} ) \in \bbS_{++}^{2n}, \\
        W_o(\sfK) &= \bbL(A_\cl(\sfK)^\textsf{T} , C_\cl(\sfK)^\textsf{T} C_\cl(\sfK)) \in \bbS_{++}^{2n}.
    \end{align*} Hence, (\ref{KM-metric}) is positive-definite.
    
    Now, we will show that the KM metric is coordinate-invariant. Let $S \in \GL_n$ and $L:= \scrT_S(\sfK)$. Then 
    \begin{equation*}
        (A_\cl(\sfL), B_\cl(\sfL), C_\cl(\sfL)) = \scrT_{\hat{S}}(A_\cl(\sfK), B_\cl(\sfK), C_\cl(\sfK)),
    \end{equation*} where $\hat{S} := \text{diag}(I_n,S)$. It follows that,
    \begin{align*}
        A_\cl(\sfL) &= \hat{S} A_\cl(\sfK) \hat{S}^{-1}, \\      B_\cl(\sfL)B_\cl(\sfL)^\textsf{T}  &= \hat{S}B_\cl(\sfK) B_\cl(\sfK)^\textsf{T} \hat{S}^\textsf{T}, \\ 
        C_\cl(\sfL)^\textsf{T} C_\cl(\sfL) &= \hat{S}^{-\textsf{T}} C_\cl(\sfK)^\textsf{T} C_\cl(\sfK) \hat{S}^{-1} \\
        W_c(\sfL) &= \hat{S} W_c(\sfK) \hat{S}^\textsf{T} \\
        W_o(\sfL) &= \hat{S}^{-\textsf{T}} W_o(\sfK) \hat{S}^{-1} 
    \end{align*} We also have,
        \begin{align*}
            \bfE(d\scrT_S(\bfV)) &= \hat{S} \bfE(\bfV) \hat{S}^{-1} \\
            \bfF(d \scrT_S(\bfV)) &= \hat{S} \bfF(\bfV) \\
            \bfG(d \scrT_S(\bfV)) &= \bfG(\bfV) \hat{S}^{-1};
        \end{align*} Plug the above into (\ref{KM-metric}) to conclude the proof.
\end{proof}

\section{Orbit Space of Output-feedback Controllers}\label{sect:rie-quo-man}

In this section, we go over key features of Riemannian quotient manifolds in order to set the machinery necessary for the proof of local linear convergence of the proposed PO for LQG in \S\ref{sect:conv-anal}. We suggest referring to \cite{boumal2023introduction,absil2008optimization} for the salient features of such a construction.

Let $\widetilde{\calM}$ be a smooth manifold with group action $\calG$. For example, the family of coordinate transformations $\{\scrT_S(\cdot) : S \in \GL_q\} \equiv \GL_q$ is a group action over $\widetilde{\calC}_q$. The orbit space of $\widetilde{\calM}$ modulo $\calG$ is the collection of all orbits $[x] := \{y \in \widetilde{\calM}:\exists g \in \calG, g(x)=y\}$ under the quotient topology:
\begin{equation*}
    \calM \equiv \widetilde{\calM}/\calG :=\{[x]:x \in \widetilde{\calM}\}.
\end{equation*} See Figure \ref{fig:quotient} for a visual depiction. We say $U \subset \calM$ is $\calG$-stable if $x \in U$ implies $[x] \subset U$. In this context, $\dim(\calM) = \dim(\widetilde{\calM}) - \dim(\calG)$; this is particularly desirable for optimization due to the reduced search space dimension.


Let $\calV_x:= \ker d \pi_x$ be the tangent space of $[x]$ at $x$. Here, $d\pi_x$ is the differential of the quotient map $\pi(x):=[x]$. Next, let $\calH_x:= \calV_x^\perp$.
Since $d\pi_x|_{\calH_x}:\calH_x \to T_{[x]}\calM$ is a bijection, we identify $\xi \in T_{[x]}\calM$ with $\lift_x(\xi):= (d\pi_x|_{\calH_x})^{-1}(\xi) \in \calH_x$. 

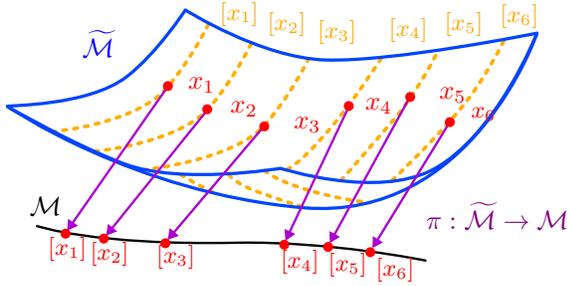
\begin{figure}[t]
    \centering
    \begin{overpic}[width=.85\linewidth]{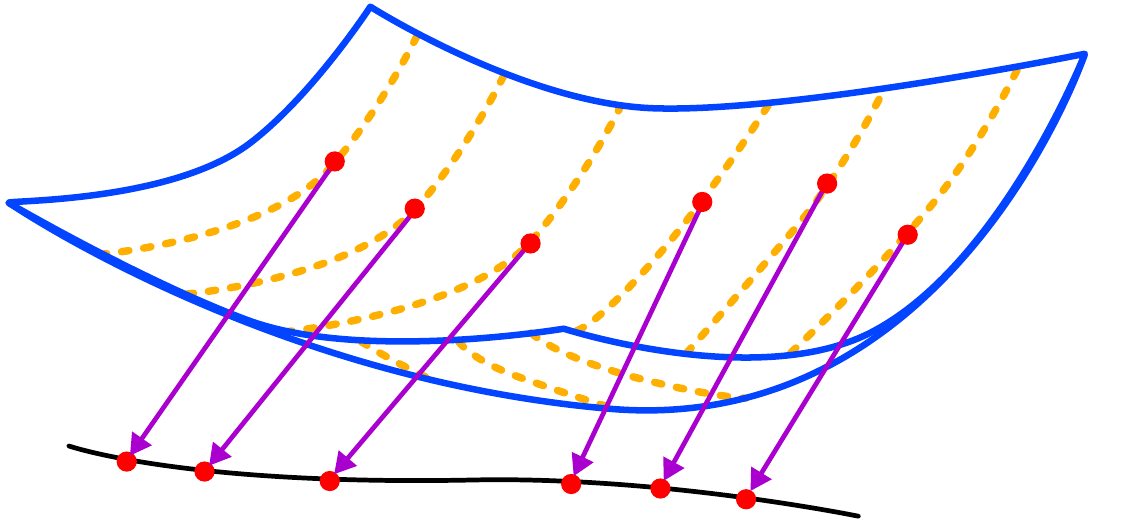}
        \put(70,68){\color{red} $x_1$}
        \put(86,59){\color{red} $x_2$}
        \put(110,53){\color{red} $x_3$}
        \put(137,60){\color{red} $x_4$}
        \put(165,64){\color{red} $x_5$}
        \put(177,57){\color{red} \small $x_6$}
        \put(17,5){\color{red} \small $[x_1]$}
        \put(33,3){\color{red} \small $[x_2]$}
        \put(58,2){\color{red} \small $[x_3]$}
        \put(105,0){\color{red} \small $[x_4]$}
        \put(123,-1){\color{red} \small $[x_5]$}
        \put(141,-3){\color{red} \small $[x_6]$}
        \put(160,15){\color{violet} $\pi:\widetilde{\calM} \to \calM$}
        \put(30,80){\color{blue} $\widetilde{\calM}$}
        \put(82,94){\color{orange} \small $[x_1]$}
        \put(100,90){\color{orange} \small $[x_2]$}
        \put(119,87){\color{orange} \small $[x_3]$}
        \put(146,88){\color{orange} \small $[x_4]$}
        \put(167,90){\color{orange} \small $[x_5]$}
        \put(188,93){\color{orange} \small $[x_6]$}
        \put(10,20){$\calM$}
    \end{overpic}
    \caption{Illustration of a manifold and its orbit space.}
    \label{fig:quotient}
\end{figure}

Convergence analysis for the proposed RGD procedure involves showing that $\calC_n^\co:=\widetilde{\calC}_n^\co/\GL_n$ is a smooth quotient manifold and inherits a Riemannian metric and retraction from the KM metric and Euclidean retraction. We state these results and omit their proofs for brevity. 

The quotient manifold structure of $\calC_n^\co$ follows from the remarkable theorem that the quotient space of minimal system realizations with $p$ inputs and $m$ outputs forms a smooth quotient manifold \cite{hazewinkel1980fine}. Since $\widetilde{\calC}_n^\co$ is open and $\GL_n$-stable, our claim is a consequence of the following result.
\begin{lemma}\label{lem:smooth-quotient-structure}
    If $\widetilde{\calM}/\calG$ is a quotient manifold and $\widetilde{\calN} \subset \widetilde{\calM}$ is open and $\calG$-stable, then $\widetilde{\calN}/\calG$ is a quotient manifold.
\end{lemma}

The inherited Riemannian structure on $\calC_n^\co$ is a result of the invariance properties of the KM metric and retraction:
\begin{align*}
    \inner{\xi, \eta}_{[\sfK]}^\KM := \inner{\lift_\sfK(\xi), \lift_\sfK(\eta)}_\sfK^\KM, \; 
    \calR_{[\sfK]}(\xi) := [\calR_\sfK(\lift_\sfK(\xi))],
\end{align*} where $\sfK \in \widetilde{\calC}_n^\co$ and $\xi,\eta \in T_{[\sfK]}\calC_n^\co$; see also~\cite[Eq. 9.27, 9.31]{boumal2023introduction}. This is the result of the following result for more  general settings.
\begin{lemma}
    Suppose $\widetilde{\calM}$ is equipped with a $\calG$-invariant Riemannian metric. Pick $x \in \widetilde{\calM}$, $g \in \calG$, and $\xi \in T_{[x]}\calM$. Then $dg_x(\lift_x(\xi)) = \lift_{g(x)}(\xi)$.
\end{lemma} 

Intuitively, this also implies (\ref{coord-equiv}) and that performing RGD over the higher-dimensional $\widetilde{\calC}_n^\co$ coincides with its performance over the lower-dimensional $\calC_n^\co$ \cite[Sect. 9.9]{boumal2023introduction}.

\section{Convergence Analysis}\label{sect:conv-anal}

In this section, we present a convergence analysis for our algorithm. Let $J_n$ and $\widetilde{J}_n$ denote the LQG cost over $\calC_n^\co$ and $\widetilde{\calC}_n$, respectively. We make the following assumption on the plant~(\ref{plant}).\footnote{A\ref{assump:lqg} states that if we perturb $\sfK^*$ along $\bfV$, then the 2nd-order approximation of $\widetilde{J}$ is unchanged if and only if $\bfV \in \calV_{\sfK^*}$. We have found that this assumption empirically holds for randomly generated plants (\ref{plant}).}
\begin{assumption}\label{assump:lqg}
    The LQG cost admits an optimal controller $\sfK^*$ that is minimal and $\ker \Hess \widetilde{J}_n(\sfK^*) = \calV_{\sfK^*}$.
\end{assumption}

\begin{theorem} \label{main-thm}
     Consider the LQG PO of plant (\ref{plant}) under A\ref{assump:lqg}. Then, there exists a neighborhood $\calU \subset \calC_n^\co$ of $[\sfK^*]$ and $L>0$ such that given $[\sfK_0] \in \calU$, the resulting sequence $([\sfK_t])_{t \geq 0}$ via $\sfK^+=\calR_{\sfK}(-\frac{1}{L}\nabla \widetilde{J}_n(\sfK))$ stays in $\calU$ and converges to $[\sfK^*]$ with a linear rate.\footnote{Explicitly, $\lim_{t \to \infty} \|\xi_{t+1}\|/\|\xi_t\| < 1$, where $\calR_{[\sfK^*]}(\xi_t)=[\sfK_t]$.}
\end{theorem}

The key idea for this analysis is to build up two conditions in order to invoke \cite[Thm. 4.19]{boumal2023introduction}. We proceed to derive the results needed for realizing this strategy. The first condition requires $[\sfK^*]$ to have positive-definite \textit{Riemannian} Hessian. A\ref{assump:lqg} ensures this via the following result.
\begin{lemma}\label{lem:hess}
    Let $\nabla \widetilde{J}(\sfK')=0$ and $(s_-,s_0,s_+)$ be the signature of $\Hess \widetilde{J}_n(\sfK')$. Then $\nabla J_n([\sfK'])=0$ and the signature of $\nabla^2 J_n([\sfK'])$ is $(s_-, s_0 - n^2, s_+)$.
\end{lemma}
\begin{proof}
    At stationary points, the signature of the Hessian is invariant of the Riemannian metric~\cite[Prop. 8.71.]{boumal2023introduction}. Thereby, $\nabla^2 \widetilde{J}_n(\sfK^*)$ and $\Hess \widetilde{J}_n(\sfK^*)$ share the same signature. By \cite[Ex. 9.46.]{boumal2023introduction}, the eigenvalues of $\nabla^2 \widetilde{J}_n(\sfK^*)$ are identical to the eigenvalues of $\nabla^2 J_n([\sfK^*])$ with $n^2$ additional zeros, thus completing the proof.
\end{proof}
We note that in the ordinary GD setup, all stationary points of $\widetilde{J}_n$ have singular Hessians; as such, the above proof is unique for the proposed Riemannian quotient manifold setup.

The second condition requires constructing a domain $\calL_0~ \subset~\calC_n^\co$ on which our RGD procedure $F(\cdot)$ is well-defined and $F(\calL_0) \subset \calL_0$. For the former, we rely on a step-size bound (Lemma \ref{lem:stability-certificate}) and proceed to choose our domain $\calL_0$ sufficiently small so that $F(\cdot)$ is well-defined. In order to show $F(\calL_0) \subset \calL_0$, we present convexity-like (Lemma \ref{lem:convexity-lemma}) and Lipschitz-like (Lemma \ref{lem:lipschitz-inequality}) inequalities, useful for analysis of first-order methods on smooth manifolds. 

Let us first demonstrate that $\calR$-balls
about $[\sfK^*]$ satisfy a strong convexity-like inequality.
\begin{lemma}\label{lem:convexity-lemma}
    Let $(\calM,\inner{.,.},\calR)$ be a Riemannian manifold. Let $f:\calM \to \bbR$ be smooth with $\nabla f(x^*)=0$ and $\nabla^2 f(x^*)~>~0$ for some $x^* \in \calM$. Then there exists $\rho>0$ for which $\nabla^2 f > 0$ on $\calD := \overline{B_{x^*}(\rho)}$, and $M>0$ such that $\calD$ contains the unique connected component $\calL_0$ of the sublevel set $\calL(f,M):=\{y \in \calM:f(y) \leq M\}$ containing $x^*$.
\end{lemma}
\begin{proof}
    Choose small enough $\rho>0$ so that $\calD \subset \text{dom}(\calR_{x^*})$ and $\nabla^2 f > 0$ on $\calD$. By \cite[Prop. 5.44]{boumal2023introduction},
    \begin{equation*}
        f(\calR_{x^*}(v)) = f(x^*) + \frac{1}{2} \inner{\nabla^2 f(x^*)v,v}_{x^*} + O(\|v\|_{x^*}^3).
    \end{equation*} Since $\nabla^2 f(x^*)>0$, then $f(\calR_{x^*}(v)) > f(x^*)$ for all $\|v\|_{x^*} \leq \rho$ for a small enough $\rho > 0$. It the follows that $f(x^*) < f(y)$ for all $y \in  \calD-\{x^*\}$.
    Pick $\epsilon > 0$ small enough such that $M:=\min f(\partial \calD)-\epsilon> f(x^*)$. Let $y \in \calL_0$ and $c:[0,1] \to \calL_0$ be any curve from $x^*$ to $y$ in $\calL_0$. If $c(t) \in \calD$ for only $0 \leq t \leq t_{\max} < 1$, then $f(c(t_{\max})) \geq M + \epsilon >M$, a contradiction. Hence, $y \in \calD$.
\end{proof}

Let $\rho>0$ be small enough so that $\widetilde{\calD} := \overline{B_{\sfK^*}(\rho)} \subset \widetilde{\calC}_n^\co$. Since $\calD := \overline{B_{[\sfK^*]}(\rho)} \subsetneq \pi(\widetilde{\calD})$, ensure that $\rho>0$ is small enough such that $\calD$ satisfies the constraints in Lemma \ref{lem:convexity-lemma}. This lemma grants us $M>0$ and $\calL_0 \subset \calD \subset \calC_n$. 

Next, we construct a lower bound on the stability radius for $\widetilde{\calC}_n^\co$ \cite{talebi2022policy}.
\begin{lemma}\label{lem:stability-certificate}
    Define the stability certificate,
    \begin{equation*}
        s(\sfK,\bfV) := (2\|A_\cl(\bfV)\|_2 \, \lambdamax(\bbL(A_\cl(\sfK),I_{2n})))^{-1}>0,
    \end{equation*} where $\lambdamax(\cdot)$ denotes the largest eigenvalue of its (symmetric) matrix argument. Then $\calR_\sfK(t\bfV) \in \widetilde{\calC}_n$ for $t \in [0,s(\sfK,\bfV))$.
\end{lemma}
\begin{proof}
    Set $P:= \bbL(A_\cl(\sfK), I_{2n})$. Then 
    \begin{equation*}
        t \, \lambdamax(A_\cl(\bfV) P + P A_\cl^\textsf{T}(\bfV)) \leq 2t \, \|A_\cl(\bfV)\|_2 \, \lambdamax(P) < 1,
    \end{equation*} and hence $t \, (A_\cl(\bfV)P + P A_\cl(\bfV)^\textsf{T}) \prec I_{2n}$.\footnote{For a pair of symmetric matrices, the notation $A \prec B$ denotes the positive-definiteness of $B-A$.} Since $A_\cl(\cdot)$ is linear and $A_\cl(\sfK)P+PA_\cl(\sfK)^\textsf{T} = -I_{2n}$, we have $A_\cl(\sfK^+)P + PA_\cl(\sfK^+)^\textsf{T} \prec 0$, where $\sfK^+= \sfK + t\bfV$.  
\end{proof}

Now, we will guarantee a Lipschitz-like inequality. Define $r(\sfK) := \frac{1}{2}\|\nabla \widetilde{J}_n(\sfK)\|_\sfK^{-1}  \min_{\|\bfV\|_\sfK = 1} s(\sfK,\bfV)$. 
\begin{lemma}\label{lem:lipschitz-inequality}
    Let $\calK \subset \widetilde{\calC}_n^\co$ be compact. Define $\calS:=\{(\sfK,\bfV) \in T \widetilde{\calC}_n^\co : \sfK \in \calK,\|\bfV\|_\sfK \leq r(\sfK)\}$. Next, define $\calS^* := \{(\scrT_S(\sfK), d\scrT_S(\bfV)) : (\sfK,\bfV) \in \calS, S \in \GL_n\}$. Then there exists $L > 0$ where for all $(\sfK,\bfV) \in \calS^*$,
    \begin{equation}\label{lip}
        \widetilde{J}_n(\calR_\sfK(\bfV)) \leq \widetilde{J}_n(\sfK) + \inner{\nabla \widetilde{J}_n(\sfK), \bfV}_\sfK + \frac{L}{2}\|\bfV\|_\sfK^2.
    \end{equation}
\end{lemma}

\begin{proof}
    Remark that $r(\cdot)$ is continuous and $\calS \subset T\widetilde{\calC}_n^\co \subset T\widetilde{\calC}_n$ is compact. Since $\widetilde{J}_n$ is also analytic over $\widetilde{\calC}_n$, then $\widetilde{J}_n \circ \calR: \calS \to \bbR$ is well-defined and analytic over compact $\calS$. As such its derivatives are bounded uniformly, and hence satisfies (\ref{lip}) \cite[Lemma 10.57]{boumal2023introduction}. 
    
    Fix $(\sfK,\bfV) \in \calS$ and $S \in \GL_n$. Since $\sfK+\bfV \in \widetilde{\calC}_n$, so is $\scrT_S(\sfK) + d\scrT_S(\bfV)$. Due to the invariance properties, (\ref{lip}) holds for $(\scrT_S(\sfK),d\scrT_S(\bfV))$.
\end{proof}

Let $L>0$ be the Lipschitz constant from Lemma \ref{lem:lipschitz-inequality} with $\calK := \widetilde{\calD}$. Ensure $L$ sufficiently large so that $\frac{1}{L} \leq \min r(\widetilde{\calD})$. The Lipschitz-like inequality holds for all $(\sfK,\bfV) \in \calS^*$. Take note that $\widetilde{\calD}^* := \pi_1(\calS^*)= \bigcup_{S \in \GL_n} \scrT_S(\widetilde{\calD}) = \pi^{-1}(\pi(\widetilde{\calD}))$. 
Set $\widetilde{\calL}_0:=\pi^{-1}(\calL_0)$; then $\widetilde{\calL}_0 \subset \widetilde{\calD}^*$.

The above analysis now leads to the following key result.
\begin{lemma}
   The set $\widetilde{\calL}_0$ in invariant under the map $\widetilde{F}$, that is, for $\sfK \in \widetilde{\calL}_0$, $\widetilde{F}(\sfK):=\calR_\sfK(-\frac{1}{L} \nabla \widetilde{J}_n(\sfK)) \in \widetilde{\calL}_0$.  
\end{lemma}
\begin{proof}
    Let $\sfK^+ := \widetilde{F}(\sfK)$. Define the curve $c(t) = \calR_\sfK(-\frac{t}{L} \nabla \widetilde{J}_n(\sfK))$. Plugging this into (\ref{lip}), it follows that $\widetilde{J}_n(\sfK^+) \leq \widetilde{J}_n(\sfK)$. 
    Furthermore, $\pi \circ c$ is continuous, contained in $\calL_0$, and connects $[\sfK]$ to $[\sfK^+]$. Since $\calL_0$ is closed, we must have $[\sfK^+] \in \calL_0$. Hence, $\sfK^+ \in \widetilde{\calL}_0$.
\end{proof}

This all implies that $F:\calL_0 \to \calL_0$, $F([\sfK]):=[\widetilde{F}(\sfK)]$ is well-defined and smooth. The local convergence guarantee in Theorem~\ref{main-thm} now follows from \cite[Thm. 4.19]{boumal2023introduction}.

\subsection{Limitation of gradient descent on LQG landscape}\label{sect:limitations}

The GD procedure lacks coordinate-equivariance (\ref{coord-equiv}), that is, $\scrT_S(\sfK - s \cdot \grad \widetilde{J}_n(\sfK)) \not = \scrT_S(\sfK) - s \cdot \grad \widetilde{J}_n(\scrT_S(\sfK))$. This is due to the fact that the Euclidean metric, despite its simplicity, is not coordinate-invariant. As such, GD has to search through $\dim(\GL_n)=n^2$ redundant dimensions. Furthermore, if we initialize $\sfK_0$ with particularly ``bad'' coordinates, one ends up with a large value of $\|\sfK_0\|_F$; in that case, $\|\grad \widetilde{J}_n(\sfK_0)\|_F$ will also be large. Such ill-conditioned coordinates in turn generally result in numerical instabilities. These two issues are resolved however when the metric is coordinate-invariant and retraction is coordinate-equivariant.

\section{Numerical Experiments and Results}\label{sect:exp}

We will now compare RGD with ordinary gradient descent (Figure \ref{fig:plots}). Our step size procedure is Algorithm \ref{alg:armijo}.
\begin{algorithm}
    \small
    \begin{algorithmic}
        \caption{Backtracking Line-Search}\label{alg:armijo}        
        \Require $\sfK \in \widetilde{\calC}_n^\co$, $\gamma \in (0,1)$, $\beta \in (0,1), \bar{s} > 0$
        \State $s \gets \bar{s}$
        \State $\sfK^+ \leftarrow K - s\nabla J_n(\sfK)$
        \While{$\sfK^+ \not \in \widetilde{\calC}_n^\co \textbf{ or } J_n(\sfK) - J_n(\sfK^+) < \gamma s \|\nabla J_n(\sfK)\|_\sfK^2$}
            \State $s \gets \beta s$
            \State $\sfK^+ \leftarrow K - s\nabla J_n(\sfK)$
        \EndWhile
        
        \Return $s$
    \end{algorithmic}
\end{algorithm}

The parameters in our algorithm were chosen as $T = 10^4$, $\gamma = 0.01$, $\beta = 0.5, \epsilon = 10^{-6}$, and $ \bar{s}=1$. We halted the simulation when $J_n(\sfK) - J_n^* < 10^{-5}$. We initialized $\sfK_0$ by generating a gain and observer with random pole placement in $(-2,-1)$. For GD, we used the same parameters and starting point. We compared GD against two KM metrics: \textbf{(1)} $c_1=c_2=c_3=1$ and \textbf{(2)} with $c_1=1$, $c_2=c_3=0$.

We ran our numerical experiments against four representative systems.\footnote{Our code is available at github.com/rainlabuw/riemannian-PO-for-LQG.} The first system is Doyle's counterexample \cite{doyle1978guaranteed}. The second system is a plant whose LQG controller is non-minimal, and the third system admits saddle points with vanishing Hessians; these systems are found in \cite{tang2021analysis}. The fourth system has dimensions $(n,m,p)=(4,3,3)$ and entries either set to zero or sampled from the standard Gaussian distribution with probability $0.8$ to promote sparsity. 

\begin{figure}
    \centering
    \includegraphics[width=.8\linewidth]{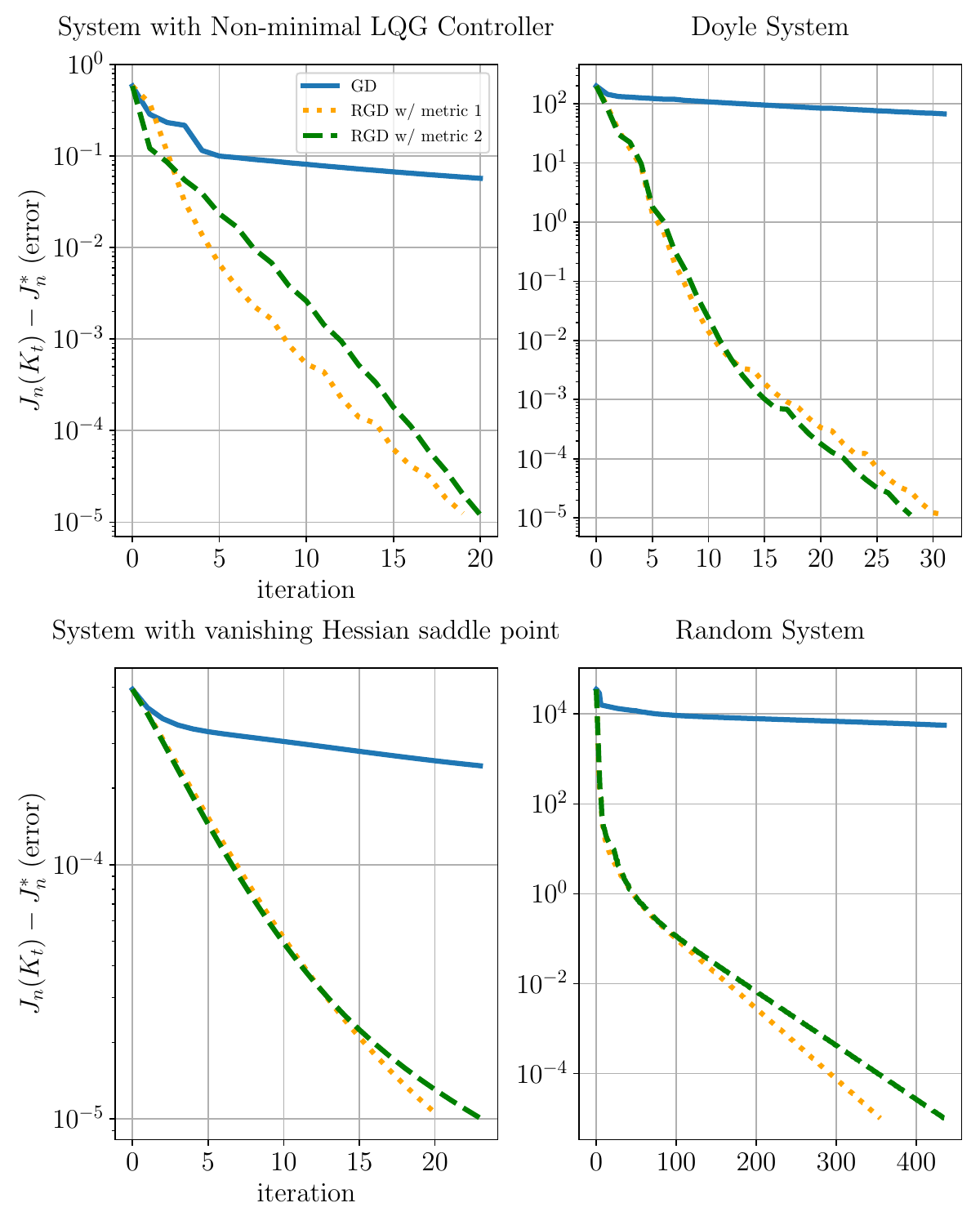}
    \caption{Comparison of RGD vs. GD for LQG PO for four distinct systems.}
    \label{fig:plots}
\end{figure}

As we observe, in all four cases, Algorithm \ref{alg:RGD} significantly outperforms GD. In fact, for the vanishing Hessian system, GD gets stuck in the non-strict saddle point. The intuition behind this behavior is that the Hessians for saddle points in our setup have $n^2$ \textit{less} zero eigenvalues than in the Euclidean case, granting RGD more leeway in avoiding the corresponding directions.

\section{Conclusion and Future Directions}\label{sect:conclusion}
In this paper, we presented a novel coordinate-invariant Riemannian metric for the space of full-order minimal dynamic output-feedback controllers. In this direction, we have shown how to minimize the LQG cost over this domain via RGD, a Riemannian gradient descent algorithm over the Riemannian quotient manifold of such controllers modulo coordinate transformation. Next, we presented the proof of guaranteed local convergence of the proposed algorithm with linear rate and computationally compared our algorithm with ordinary GD for four representative systems. 
Future directions include PO constrained synthesis, and exploring other coordinate-invariant Riemannian metrics induced by system-theoretic constructs.

\section{Acknowledgements}
The authors thank Shahriar Talebi for many helpful discussions on policy optimization and Riemannian geometry of stabilizing feedback gains.
The research of the authors has been supported by NSF grant ECCS-2149470 and AFOSR grant FA9550-20-1-0053.

\bibliographystyle{ieeetr}
\bibliography{references}

\end{document}